\pdfoutput=1
\documentclass[12pt]{article}
\usepackage[utf8]{inputenc}
\usepackage{color}
\usepackage[dvipsnames]{xcolor}
\usepackage{amsmath}
\usepackage{parskip}
\usepackage[showframe=false]{geometry}
\usepackage{changepage}
\usepackage{graphicx}
\usepackage{tikz}
\usepackage{ mathdots }
\usetikzlibrary{arrows.meta}
\usetikzlibrary{patterns}
\usepackage{ dsfont }
\usepackage{diagbox}
\usepackage[refpage]{nomencl}
\usepackage{amsthm}
\usepackage[compact]{titlesec}         
\titlespacing{\section}{10pt}{10pt}{10pt} 
\AtBeginDocument{
  \setlength\abovedisplayskip{10pt}
  \setlength\belowdisplayskip{10pt}}
\usepackage{tikz}
\usetikzlibrary{positioning}
\usepackage{amssymb}
\usepackage{tikz}
\usetikzlibrary{patterns}
\usetikzlibrary{decorations.pathmorphing}
\usetikzlibrary{angles,quotes}
\tikzset{snake it/.style={decorate, decoration=snake}}
\usepackage{cancel}
\usepackage{caption}
\usepackage{subcaption}
\usepackage{ textcomp }
\makeindex
\usepackage{ gensymb }
\usepackage{float}
\usepackage{rotating}
\usepackage{tabularx}
\usepackage{delarray}

\newcommand{\tpmod}[1]{{\@displayfalse\pmod{#1}}}

\newcommand{\Sym}[1]{\mathrm{Sym}({#1})}

\usepackage{amsmath,xcolor}

\usepackage{mathtools}
\usepackage[shortlabels]{enumitem}

\usepackage[english]{babel}
\usepackage{soul}
\newtheorem{theorem}{Theorem}[section]
\newtheorem{corollary}[theorem]{Corollary}
\newtheorem{lemma}[theorem]{Lemma}
\newtheorem{definition}[theorem]{Definition}

\theoremstyle{definition}

\usepackage{tikz}
\title{\bf{Finite Coxeter Groups and Generalized Elnitsky Tilings}}
\author{Robert Nicolaides and Peter Rowley }

\begin{document}

\maketitle
\begin{abstract}
In \cite{elnitskyThesis}, Elnitsky constructed three elegant bijections between classes of reduced words for Type $\mathrm{A}$, $\mathrm{B}$ and $\mathrm{D}$ families of Coxeter groups and certain tilings of polygons. This paper presents a particular generalization of this concept to all finite Coxeter groups in terms of embeddings into the symmetric group, which combined with \cite{nicolaidesrowleyBruhat} and \cite{nicolaidesrowleyDeletion} enables tilings to be constructed for all finite Coxeter groups. 
\end{abstract}

\section{Introduction}

Reduced words for elements of a Coxeter group have been extensively investigated, as they are implicated in a host of problems and questions in many diverse areas. For example, fully commutative elements, which are elements such that any reduced word can be obtained from another by commutation of generators, illustrate this vividly. They arise in applications to symmetric functions associated with the Weyl groups of type $\mathrm{B}$ and $\mathrm{D}$ (see \cite{stembridge} and references therein). Further they also give a natural indexing of a basis for the generalized Temperley-Lieb algebras (see \cite{biagiolajouketnadeau}). While fully commutative elements in $\mathrm{Sym}(n)$ have connections to permutations which avoid the pattern $321$ (in one-line notation) and the Catalan numbers. As to the number of reduced words for an element of  $\mathrm{Sym}(n)$, Stanley in \cite{stanley} demonstrated that this can be calculated in terms of particular shaped Young tableaux. We note that this description does not, in general, allow easy enumeration. However, for commutation classes of the longest element in $\mathrm{Sym}(n)$ see \cite{denoncortdanadustin}.

A different perspective on reduced words, expounded in  Elnitsky's PhD dissertation \cite{elnitskyThesis}, lead into the realms of certain types of tilings of $m$-gons. In \cite{elnitsky}  rhombic tilings of $2n$-gons are shown to be in bijection with commutation classes of reduced words in Coxeter groups of type $\mathrm{A}_{n-1}$. Further, similar types of tilings are also given there for types $\mathrm{B}$ and $\mathrm{D}$ Coxeter groups. For a selection of related work see \cite{epty}, \cite{hamanaka}, \cite{nicolaidesrowley}, \cite{tenner1} and \cite{tenner2}.

In this paper we introduce generalized Elinitsky tilings, which extend the work in \cite{elnitsky}.  The central concept is that of an Elnitsky embedding which we define shortly. The main results of this paper are contained in Theorems \ref{lemma E-embedding} and \ref{maintheorem} which show that Elnitsky embeddings lead to tilings. Two further papers \cite{nicolaidesrowleyBruhat} and \cite{nicolaidesrowleyDeletion} then demonstrate how to obtain Elnitsky embeddings for any finte Coxeter group, so yielding tilings for all these groups. In fact, by using faithful permutation representations on the cosets of standards parabolic subgroups of the finite Coxeter group a cornucopia of tilings are revealed (see Theorem 1.1 of \cite{nicolaidesrowleyBruhat}). We note that \cite{nicolaidesrowleyBruhat} also includes a number of explicit examples of tilings, including some for the iconic Coxeter group $\mathrm{E}_8$.

From now on $W$ denotes a finite Coxeter group with $S$ being the fundamental reflections of $W$. For $u, v \in W$, $u <_R v$ denotes the (right) weak order and $u <_B v$ denotes the Bruhat order on $W$. We now define an Elnitsky embedding, or E-embedding for short, as follows.

\begin{definition}\label{def e-embedding} Suppose that $\varphi : W \hookrightarrow \Sym{n}$ is an embedding (meaning an injective homomorphism). Then $\varphi$ is an E-embedding for $W$ if for all $u,v \in W$
$$ u <_R v \text{ implies } \varphi(u) <_B \varphi(v).$$
\end{definition}

Suppose $\varphi$ is an embedding of $W$ into $\Sym{n}$, and let $w \in W$. Also we choose $\alpha \in (0,\pi/2)$. From this data we construct a $2n$-polygon, where the left-hand $n$ edges all have the same length and pairwise they subtend an angle of $\frac{\pi - 2\alpha}{n-1}$. The remaining $n$ edges of the polygon are determined by the permutation $\varphi(w)$, and we denote this part of the polygon (or half-polygon) by $\mathrm{B}^\alpha_\varphi(w)$. See Definition \ref{Defn Boarder} and also Section \ref{Bijection} where we discuss how we may vary the subtending angles.

In Section \ref{Bijections} we show that given a fixed element $w \in W$,  these embeddings naturally associate each reduced word $w = s_{i_1}\ldots s_{i_\ell}$ to a tiling, denoted $\mathrm{T}_\varphi(s_{i_1}s_{i_2}\ldots s_{i_\ell})$, of some  polygon corresponding to $w$, $\mathrm{P}_{\varphi}(w)$. We denote the set of all such tilings by $\mathcal{T}_{\varphi}(w)$.

Theorem \ref{lemma E-embedding} is the key observation relating E-embeddings to tilings of polygons. In the theorem we use $\prec$ to denote that the half polygon of the one permutation does not cross the half polygon of the other permutation - see Definition \ref{precident}  and Section \ref{Bijections} for more details.

\begin{theorem}\label{lemma E-embedding}
Suppose $\varphi : W \hookrightarrow \Sym{n}$ is an E-embedding and $\alpha \in [\pi/4,\pi/2)$. Then for all $w \in W$ and all reduced expressions $w = s_{i_1}s_{i_2}\ldots s_{i_\ell}$, 
$$\mathrm{B}^\alpha_\varphi(id) \prec \mathrm{B}^\alpha_\varphi(s_{i_1})\prec \mathrm{B}^\alpha_\varphi(s_{i_1}s_{i_2})\prec \ldots \prec \mathrm{B}^\alpha_\varphi(s_{i_1}\ldots s_{i_\ell}). $$
\end{theorem}

We let $J_{\varphi}$ be the set consisting of pairs of elements in $S$ such that for all reduced words in $W$, applying the braid moves for each pair preserves the corresponding tiling, and let $\mathcal{R}_{J_{\varphi}}(w)$ be the equivalence classes of the reduced words for $w$ under these moves. Then we have

\begin{theorem}\label{maintheorem} Suppose that $\varphi$ is a E-embedding of $W$. Then for all $w \in W$ there exists a bijection between $\mathcal{T}_{\varphi}(w)$ and $\mathcal{R}_{J_{\varphi}}(w)$.
\end{theorem}

If $\varphi$ is an embedding of $W$ in $\Sym{n}$ and $s \in S$, we define the \emph{support interval} of $s$ to be $${\mathrm{I}}_\varphi(s) = \bigcup_{i=1}^k \{a_i, a_i+1,\ldots, b_i-1,b_i\} $$ where $ \varphi(s) = \prod_{i=1}^k(a_i,b_i)$ with $a_i<b_i$.
In Lemma~\ref{Jdescription} we can pinpoint $J_{\varphi}$ in terms of support intervals.

\begin{lemma}\label{Jdescription} Let $\varphi : W \hookrightarrow \Sym{n}$ be a E-embedding. Then $$J_{\varphi} = \{\{r,s\}\subseteq S \, | \, \mathrm{I}_\varphi(s) \cap \mathrm{I}_\varphi(r) = \emptyset\}.$$
\end{lemma}

 The weak order has a succinct interpretation in the context of Elnitsky's tiling for $W = \Sym{n}$. Namely  for all $w\in W$, $s \in S$ we have $w <_R ws$ if and only if the polygon $\mathrm{P}(ws)$ is obtained by appending a rhombus to $\mathrm{P}(w)$ (see Theorem \ref{Elnitsky A}).  What about the Bruhat order? Is this reflected visually somehow in these tilings also? Here we take inspiration from Elnitsky's observation that having an angle of each edge being greater than $\pi/3$ from the horizontal avoids certain self-intersections of tiles. We will show that so long as the angles are always at least $\pi/4$ from the horizontal, self-intersections do not occur for comparable elements in the Bruhat order and hence produce bona fide tiles.  We prove this in Theorem~\ref{symmetric} which is also an important ingredient in the proof of Theorem \ref{lemma E-embedding}.  
 
\begin{theorem}\label{symmetric} For all $u, v \in W = \Sym{n}$,  $u <_B  v$ implies $\mathrm{B}^{\pi/4}(u) \prec  \mathrm{B}^{\pi/4}(v)$.
\end{theorem}

 Interestingly, we observe that the converse is not true, that is, non-comparable elements in the Bruhat order can yield tiles in the same sense. We highlight the apparent rarity of this phenomenon as, for example in $\Sym{4}$, amongst the 87 pairs of non-comparable elements there is only one pair (up to inverses)  exhibiting this behaviour. See the beginning of Section \ref{Bijection}.

Section~\ref{background} mostly looks at the tile constructions for type $\mathrm{D}$ Coxeter groups which appear in \cite{elnitsky}, giving some explicit examples to set the scene for the E-polygons which are introduced in Section~\ref{section E-polygons}. There, as a prelude to Theorem~\ref{symmetric}, exhaustive examples are given for $\Sym{3}$. The main task in Section~\ref{Bijection} is the proof of Theorem~\ref{symmetric}.

\section{Background}\label{background}

In this section, we first give a quick review of some material on Coxeter groups which we need. Then to place this paper in context we briefly describe some of Elnitsky's tiling constructions. A finite Coxeter group $W$ with $S$ its set of fundamental reflections means that
$$W= \langle S \; | \; (rs)^{m_{r,s}} = 1 \;r, s \in S \rangle,$$
where $m_{r,r} = 1$ and $m_{r,s}  > 1$ for all $r\not = s$. We use $\ell(w)$ to denote the length of $w$ in $W$.

Let $T$ denote the set of reflections of $W$, and let $w_1, w_2 \in W$. The weak (right) order on $W$, $<_R$, is the transitive extension of $w_1 <_R w_2$ if $w_2 = w_1s$ for some $s \in S$ and $\ell(w_2) = \ell(w_1)+1$. While the Bruhat order on $W$ is the transitive extension of $w_1 <_B w_2$ if $w_2 = w_1t$ for some $t \in T$ and $\ell(w_2) = \ell(w_1)+1$. The Bruhat order has been studied extensively, but here we only have need of a characterization of this order in the symmetric group.

\begin{theorem}\label{BruhatSym} Let $w \in W = \Sym{n}$, and let $t = (a,b) \in T$ with $a < b$. Then $w <_B wt$ if and only if $aw^{-1} < bw^{-1}$.
\end{theorem}

\begin{proof} See Lemma 2.1.4 of \cite{BB} (note that we write permutations on the right here).
\end{proof}

When $W$ is the Coxeter group of type $\mathrm{A}_{n-1}$, for each $w \in W$ Elnitsky constructs a tiling of a $2n$-gon. We say more on this case after jumping straight into the case of type $\mathrm{D}_n$, as this illustrates the ``overlapping tiles" issue.

Let $(W,S)$ be $(\mathrm{D}_n,\{s_1,s_2,\ldots,s_n\})$, a standard embedding of $\mathrm{D}_n$ into $\Sym{2n}$ with $s_1=(1,-2)(2,-1)$ and $s_i = (i-1,i)(-(i-1),-i)$ for each $i \in \{2, \dots, n \}$. In this setting, we construct polygons $\mathrm{P}(w)$ for all $w \in W$as follows.
\begin{enumerate}[$(i)$]
    \item Let $\mathrm{U}$ be the upper-most vertex of our $4n$-gon, and $\mathrm{L}$ the lower-most vertex and $M$ the vertex that is an equal distance from both.
    \item Let the first $2n$ edges anticlockwise from $\mathrm{U}$ be those of the regular $4n$-gon with unit length edges, labelling them with $\{n,n-1,\ldots,1,-1,\ldots,-n\}$ respectively.
    \item  For each $i \in \{-n,-n+1,\ldots,-1,1,\ldots,n\}$, construct and label the $2n$ edges anti-clockwise from $\mathrm{L}$  (in order $1, 2, \dots, n, -n, -n-1, \dots, -1$) such that the $i^{th}$ edge from $M$ in this direction has the same length of that $(i)w^{-1}$ and parallel to it also. 
\end{enumerate}
As observed by Elnitsky, in order to avoid intersections of tiles the absolute gradient of each of our edges from the horizontal must  be strictly greater than $\pi/3$. We present an example of such a polygon for $\mathrm{D}_4$.


\begin{figure}[H]
    \centering
    \includegraphics[width = 6cm ]{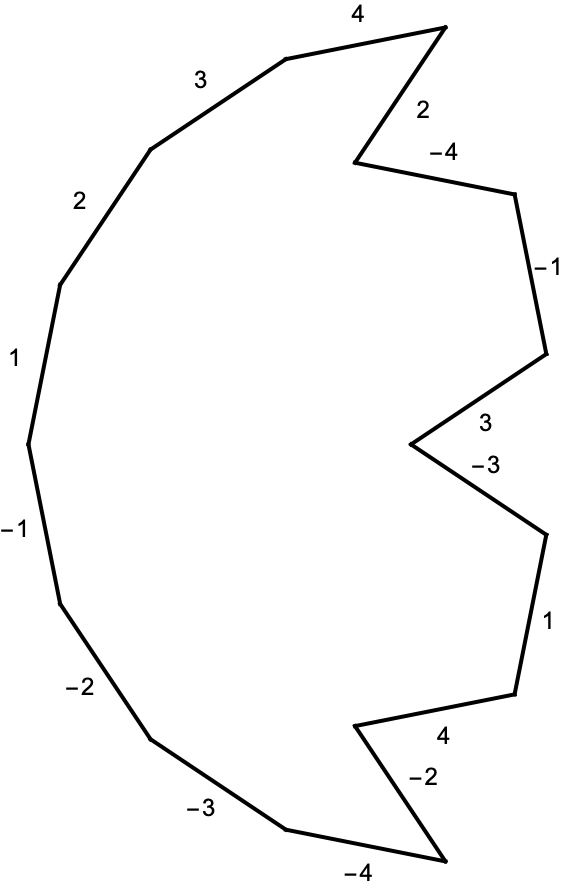}
    \caption{$\mathrm{P}(w)$ for $w = (\overset{+}{1}, \overset{-}{3}, \overset{+}{4}, \overset{-}{2})  \in \mathrm{D}_4$ .}
    \label{border}
\end{figure}

 The tiles used for type $\mathrm{D}$ are more complex than those for type $\mathrm{A}$. In particular, we have a set of \emph{megatiles} at our disposal. The megatiles are a subset of octagons with unit edge lengths which can be constructed as follows. Its upper-most vertex, $\mathrm{U}_0$, and lower-most vertex, $\mathrm{L}_0$, must lie on a vertical line. Its first four edges anti-clockwise from $\mathrm{U}_0$ must be symmetric through the horizontal line passing through the middle vertex. Call these edges $E_0$. Then to make the remaining edges perform the following on $E_0$:
\begin{enumerate}[$(i)$]
    \item Reflect $E_0$ through the vertical line passing through $\mathrm{U}_0$ and $\mathrm{L}_0$. 
    \item Transpose the first and second pair of edges and the third and fourth pair of edges in $E_0$ respectively.

\end{enumerate}

We present two different tilings of the polygon for $w = (\overset{+}{1}, \overset{-}{3}, \overset{+}{4}, \overset{-}{2}) \in \mathrm{D}_4$.
\begin{figure}[H]
     \centering
     \begin{subfigure}[b]{0.4\textwidth}
         \centering
         \includegraphics[width=\textwidth]{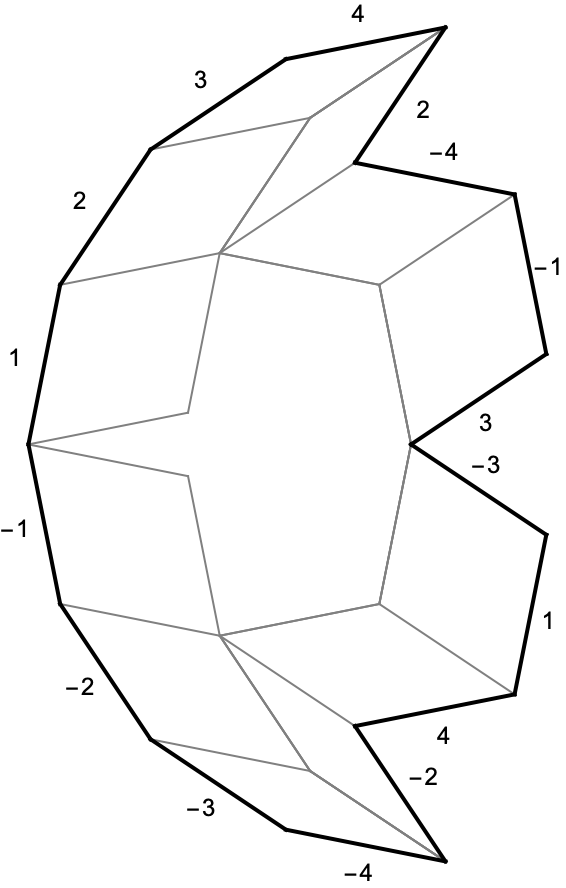}
         \caption{$w= s_1 s_2 s_3 s_4 s_1 s_2 s_3 $}
      
     \end{subfigure}
     \hspace{2cm}
     \begin{subfigure}[b]{0.4\textwidth}
         \centering
         \includegraphics[width=\textwidth]{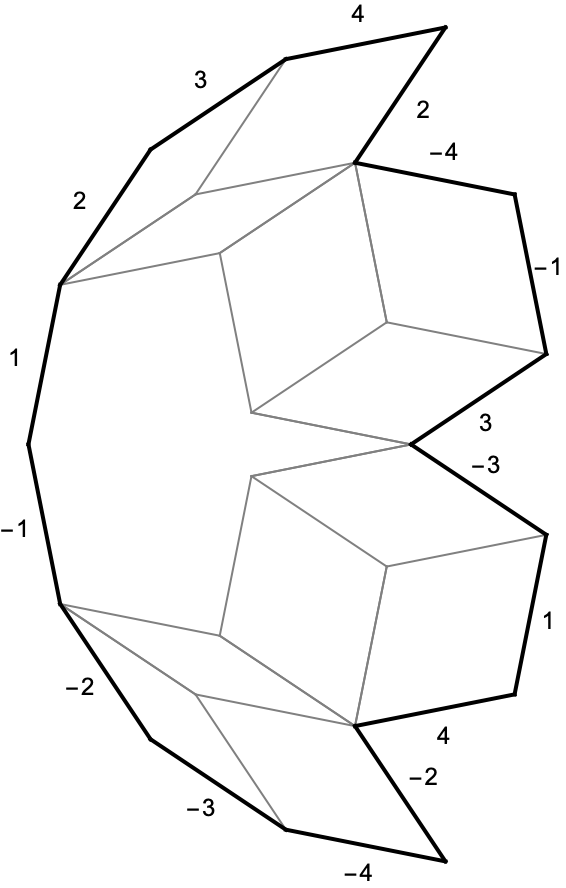}
         \caption{$w= s_2 s_1 s_2 s_4 s_2 s_3 s_2 $}
         
     \end{subfigure}
     \caption{Two Elnitsky Tilings for $w = (\overset{+}{1}, \overset{-}{3}, \overset{+}{4}, \overset{-}{2}) \in \mathrm{D}_4$ }
     \label{DtilingsFigure}
\end{figure}


Figure \ref{DtilingsFigure} is part of a general phenomenon captured by Theorem 7.1 of Elnitsky \cite{elnitsky}. This result essentially gives a bijection between reduced words of $\mathrm{D}_n$ and tilings by rhombi and megatiles (up to commutation classes).

Elnitsky notices that these tilings may have self-intersections. Luckily, as noted earlier, he also provides a remedy for these intersections: if the angles from the horizontal of each edge in the border are at least $\pi/3$ then these self-intersections are removed. We present an example in Figure \ref{Figure Dn self-intersect} for the reduced word  $s_3s_4s_5s_2s_3s_4s_5s_1s_2s_3s_1$ in $\mathrm{D}_5$.

\begin{figure}[H]
     \centering
     \begin{subfigure}[b]{0.4\textwidth}
         \centering
         \includegraphics[height=8cm]{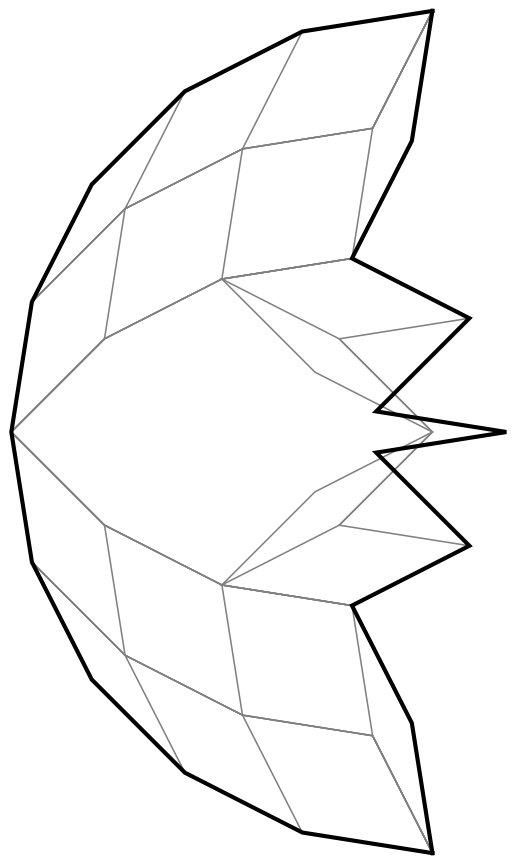}
     \end{subfigure}
     \hspace{2cm}
     \begin{subfigure}[b]{0.4\textwidth}
         \centering
         \includegraphics[height=8cm]{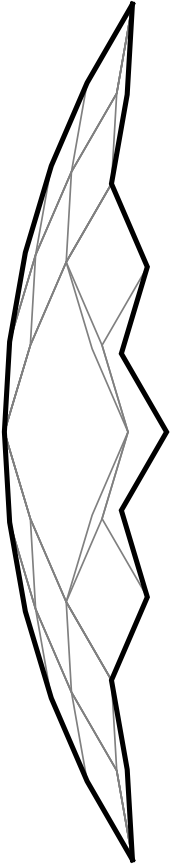}
     \end{subfigure}
     \caption{Elnitsky's Type $\mathrm{D}$ tilings for the reduced word $s_3s_4s_5s_2s_3s_4s_5s_1s_2s_3s_1$ using the usual round representation (left) and ensuring that the angles are each $\pi/3$ from the horizontal (right).}
      \label{Figure Dn self-intersect}
\end{figure}


For $W = \Sym{n}$ and $w \in W$, $\mathrm{P}(w)$ is a $2n$-gon with $n$ edges labelled 1, \dots, n clockwise from $\mathrm{L}$, the lowest-most vertex. Analogous to the $\mathrm{D}_n$ case, starting from $\mathrm{L}$, going anticlockwise the $i^{th}$ edge from $\mathrm{L}$ is parallel to, and labelled as,  $(i)w^{-1}$. Elnitsky also proved the following result, analogous to the result for type $\mathrm{D}$ where it is only necessary to tile with rhombi. Here we write $\mathrm{T}(w)$ and $\mathcal{R}_{\mathrm{J}}(w)$ for $\mathrm{T}_\varphi(w)$ and $\mathcal{R}_{\mathrm{J}_\varphi}(w)$.

 \begin{theorem}\label{Elnitsky A}
 Suppose $W = \Sym{n}$ and
  let   $\mathrm{J} = \{\{s_i,s_j\} \;| \; |i-j|\ge 2 \}$.
  Then for  all $w \in W$ there exists a bijection between $\mathrm{T}(w)$ and $\mathcal{R}_{\mathrm{J}}(w)$.
\end{theorem}

\section{$\mathrm{E}$-polygons}\label{section E-polygons}
Let $(W,S) = (\Sym{n},\{s_1,\ldots,s_{n-1}\})$ and fix some $\alpha \in (0,\pi/2)$.  

\begin{definition}\label{Defn edge Set}
Let $\beta_n^k(\alpha)$ denote the 2-dimensional, real, unit vector $$\beta_n^k(\alpha) = \begin{pmatrix}
-\cos\left(\frac{(k-1)(\pi-2\alpha)}{n-1}+\alpha\right)\\
\phantom{-}\sin\left(\frac{(k-1)(\pi-2\alpha)}{n-1}+\alpha\right)
\end{pmatrix}$$ for $k = 1,\ldots,n$.  We will sometimes refer to upper and lower entries of the vectors as the $x$ and $y$ coordinates respectively. We call $\mathcal{B}^\alpha_n = \{\beta_n^1(\alpha), \ldots,\beta_n^n(\alpha)\}$ a set of \textit{underlying vectors}.
\end{definition}
When $\alpha$ is clearly fixed from context, we will write $\beta_n^k(\alpha)$ more simply as $\beta_n^k$. Visually, these vectors are distributed evenly on the upper half of the unit circle whose absolute angles from the horizontal axis is at least $\alpha$, see Figure \ref{Figure Betanis}. In practice, the angles do not need to be evenly distributed - we just need the angle of $\beta_n^i$ measured anti-clockwise from $\begin{pmatrix}
1\\
0
\end{pmatrix}$ to be greater than that of $\beta_n^j$ whenever $i>j$.

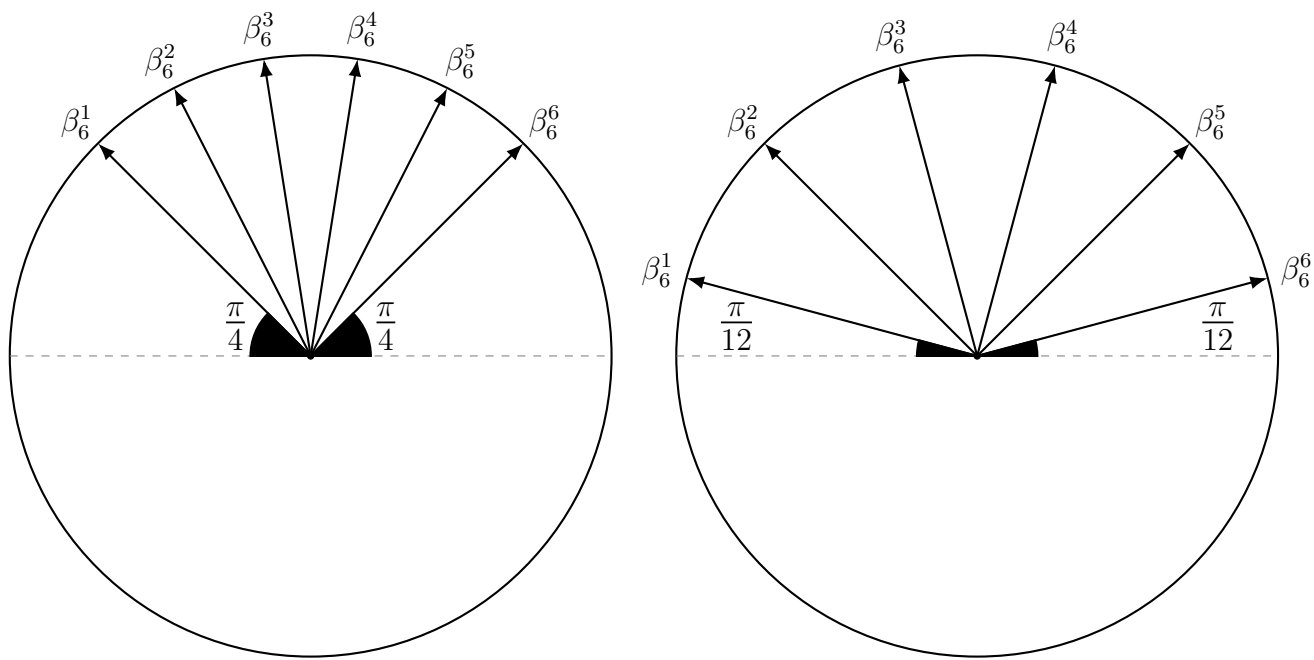
\begin{figure}[H]
    \hspace{-1cm}
\begin{subfigure}[b]{-0.1\textwidth}
     \begin{tikzpicture}[scale=0.8]
    \draw[black,thick] (5,0) arc (360:0:5);  
    \draw[gray,dashed] (-5,0) -- (5,0);  
    \draw[gray,dashed] (0,0) -- (3.5355,3.5355);  
    \draw[gray,dashed] (0,0) -- (-3.5355,3.5355); 
    \draw[scale=5,black,-Latex,thick] (0,0) -> (0.707, 0.707);  
    \node (a) at (5.5*0.707, 5.5*0.707) { $\beta_6^6$};
    \draw[scale=5,black,-Latex,thick] (0,0) -> (0.454, 0.891);  
    \node (a) at (5.5*0.454, 5.5*0.891) { $\beta_6^5$};
    \draw[scale=5,black,-Latex,thick] (0,0) -> (0.156, 0.988);  
    \node (a) at (5.5*0.156, 5.5*0.988) { $\beta_6^4$};
    \draw[scale=5,black,-Latex,thick] (0,0) -> (-0.156, 0.988); 
    \node (a) at (-5.5*0.156, 5.5*0.988) { $\beta_6^3$};
    \draw[scale=5,black,-Latex,thick] (0,0) -> (-0.454, 0.891); 
    \node (a) at (-5.5*0.454, 5.5*0.891) { $\beta_6^2$};
    \draw[scale=5,black,-Latex,thick] (0,0) -> (-0.707, 0.707);
    \node (a) at (-5.5*0.707, 5.5*0.707) { $\beta_6^1$};
    \filldraw (0,0) circle[black,radius=1.5pt];
    \filldraw[black,thick] (0,0)--(1,0) arc (0:45:1)--cycle; 
    \node (a) at (1.25,0.5) {\Large $\frac{\pi}{4}$};
    \filldraw[black,thick] (0,0)--(-1,0) arc (0:-45:-1)--cycle; 
    \node (a) at (-1.25,0.5) {\Large $\frac{\pi}{4}$};
    \end{tikzpicture}   
\end{subfigure}
\hspace{8cm}
\begin{subfigure}[b]{-0.1\textwidth}
    \begin{tikzpicture}[scale=4]
    \draw[black,thick] (1,0) arc (360:0:1);  
    \draw[gray,dashed] (-1,0) -- (1,0);  
    \draw[gray,dashed] (0,0) -- (-0.965926, 0.258819);  
    \draw[gray,dashed] (0,0) -- (0.965926, 0.258819);

    \draw[scale=1,black,-Latex,thick] (0,0) -> (-0.965926, 0.258819); 
    \node (a) at (-0.965926*1.1, 0.258819*1.1) { $\beta_6^1$};
    \draw[scale=1,black,-Latex,thick] (0,0) -> (0.965926, 0.258819); 
    \node (a) at (0.965926*1.1, 0.258819*1.1) { $\beta_6^6$};
    
    \draw[scale=1,black,-Latex,thick] (0,0) -> (-0.707107, 0.707107);
    \node (a) at (-0.707107*1.1, 0.707107*1.1) { $\beta_6^2$};
    \draw[scale=1,black,-Latex,thick] (0,0) -> (0.707107, 0.707107); 
    \node (a) at (0.707107*1.1, 0.707107*1.1) { $\beta_6^5$};
    
    \draw[scale=1,black,-Latex,thick] (0,0) -> (-0.258819, 0.965926);  
     \node (a) at (-0.258819*1.1, 0.965926*1.1) { $\beta_6^3$};
    \draw[scale=1,black,-Latex,thick] (0,0) -> (0.258819, 0.965926); 
    \node (a) at (1.1*0.258819, 1.1*0.965926) { $\beta_6^4$};
    \filldraw (0,0) circle[black,radius=0.3pt];
    \filldraw[black,thick] (0,0)--(0.2,0) arc (0:15:0.2)--cycle; 
    \node (a) at (0.8,0.1) {\Large $\frac{\pi}{12}$};
    \filldraw[black,thick] (0,0)--(-0.2,0) arc (0:-15:-0.2)--cycle; 
    \node (a) at (-0.8,0.1) {\Large $\frac{\pi}{12}$};
    \end{tikzpicture}
\end{subfigure}

    \caption{The set $\mathcal{B}^\alpha_6$ for $\alpha = \pi/4$ and $\pi/12$ respectively.}
    \label{Figure Betanis}
\end{figure}

\begin{definition}\label{Defn w-images}
For all $w \in W$ we define the ordered set $$\mathcal{B}^\alpha_n(w):=\{\beta_n^{(1)w^{-1}},\ldots \beta_n^{(n)w^{-1}}\}$$ to be the $w$-image of  $\mathcal{B}^\alpha_n$. 
\end{definition}

\begin{definition}\label{Defn Boarder}
Given $w \in W$, for $i=1,\ldots,n$, we define $\mathrm{B}_n^{\alpha}(w)_i$ to be the unit length line segment whose end points are $\Sigma_{j=1}^{i-1}\mathcal{B}^\alpha_n(\sigma)_j$ and $\Sigma_{j=1}^{i}\mathcal{B}^\alpha_n(\sigma)_j$.  Here it is understood that $\Sigma_{j=1}^{0}\mathcal{B}^\alpha_n(\sigma)_j$ is the zero-vector. We call $$\mathrm{B}_n^\alpha(w) = \bigcup_{i = 1}^n \mathrm{B}_n^{\alpha}(w)_i$$ the \textit{border} of $w$ and
$\mathrm{B}_n^{\alpha}(w)_i$ it's edge in the $i^{th}$ position.
\end{definition}

The borders $\mathrm{B}_n^\alpha(id)$ and $\mathrm{B}_n^\alpha((1,2)(4,5,6))$ in $\Sym{6}$ are displayed in Figure \ref{Figure borders Examples}:

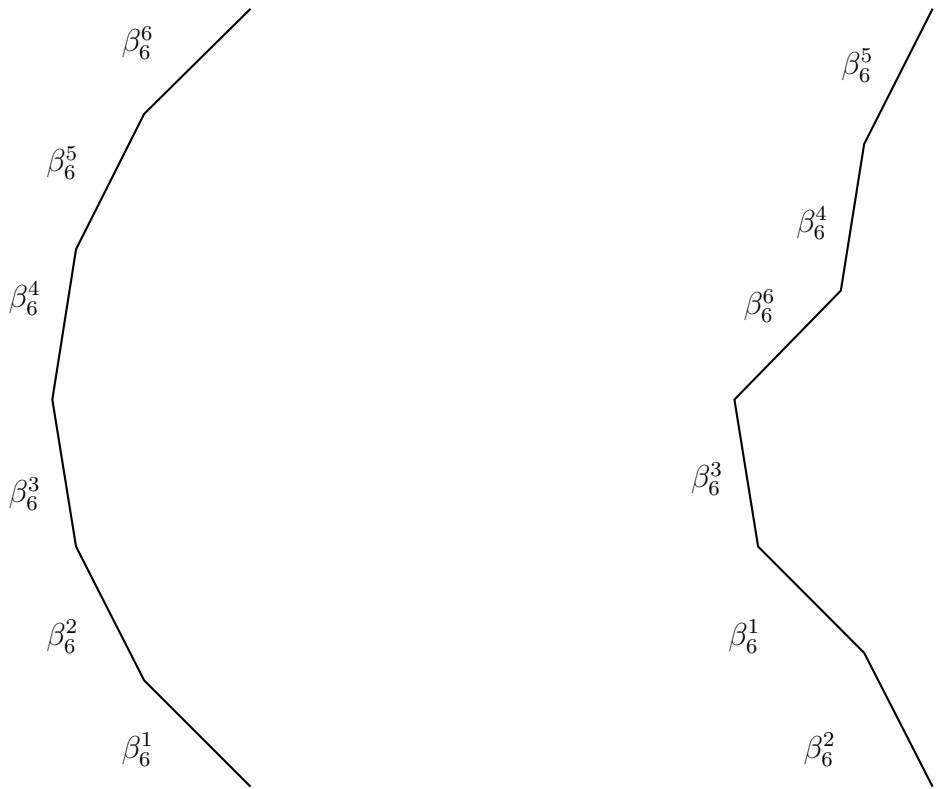
\begin{figure}[H]
\begin{subfigure}[b]{-0.1\textwidth}
   \begin{tikzpicture}
   \draw[black,thick,scale=0.5] 
   (0,0)--(-2.83,2.83)--(-4.64,6.39)--(-5.27,10.3)--(-4.64,14.3)--(-2.83,17.9)--(0.*10^-2,20.7);
   \node (a) at (-1.5,0.5) { $\beta_6^1$};
   \node (a) at (-2.5,2) { $\beta_6^2$};
   \node (a) at (-3,3.9) { $\beta_6^3$};

   \node (a) at (-3,6+0.5) { $\beta_6^4$};
   \node (a) at (-2.5,6+2.3) { $\beta_6^5$};
   \node (a) at (-1.5,6+3.9) { $\beta_6^6$};
   \end{tikzpicture}    
\end{subfigure}
\hfill
\begin{subfigure}[b]{0.4\textwidth}
   \begin{tikzpicture}
   \draw[black,thick,scale=0.5] 
(0, 0)--(-1.82, 3.56)--(-4.64, 6.39)--(-5.27, 10.3)--(-2.44,13.2)--(-1.82, 17.1)--(0, 20.7);
   \node (a) at (-1.5,0.5) { $\beta_6^2$};
   \node (a) at (-2.5,2) { $\beta_6^1$};
   \node (a) at (-3,4.1) { $\beta_6^3$};
   \node (a) at (-1.6,6+1.5)  { $\beta_6^4$};    
   \node (a) at (-1,6+3.6){ $\beta_6^5$};
   \node (a) at (-2.3,6+0.4) { $\beta_6^6$};
   \end{tikzpicture}    
\end{subfigure}
   \caption{The borders $\mathrm{B}_n^{\alpha}(id)$ and $\mathrm{B}_n^{\alpha}((1,2)(4,5,6))$ in $\Sym{6}$ with $\alpha = \pi/4$.}
   \label{Figure borders Examples}
\end{figure}

\begin{definition}
For all $u,v\in \Sym{n}$, we define the \emph{$\mathrm{E}$-polygon} of $(u,v)$ (with respect to $n$ and $\alpha$), denoted $\mathrm{P}_n^\alpha(u,v)$, to be the $2n$-gon formed from the union of $\mathrm{B}_n^\alpha(u)$ and $\mathrm{B}_n^\alpha(v)$: $$\mathrm{P}_n^\alpha(u,v) = \mathrm{B}_n^\alpha(u) \bigcup \mathrm{B}_n^\alpha(v).$$ 
\end{definition}

Consequently, $\mathrm{P}_n^\alpha(u,v) = \mathrm{P}_n^\alpha(v,u)$. If $u = id$, we simplify $\mathrm{P}_n^\alpha(u,v)$ to $\mathrm{P}_n^\alpha(v)$. 

\begin{figure}[H]\vspace{-2cm}
    \hspace{-2cm}
    \includegraphics[width = 18cm, angle =0]{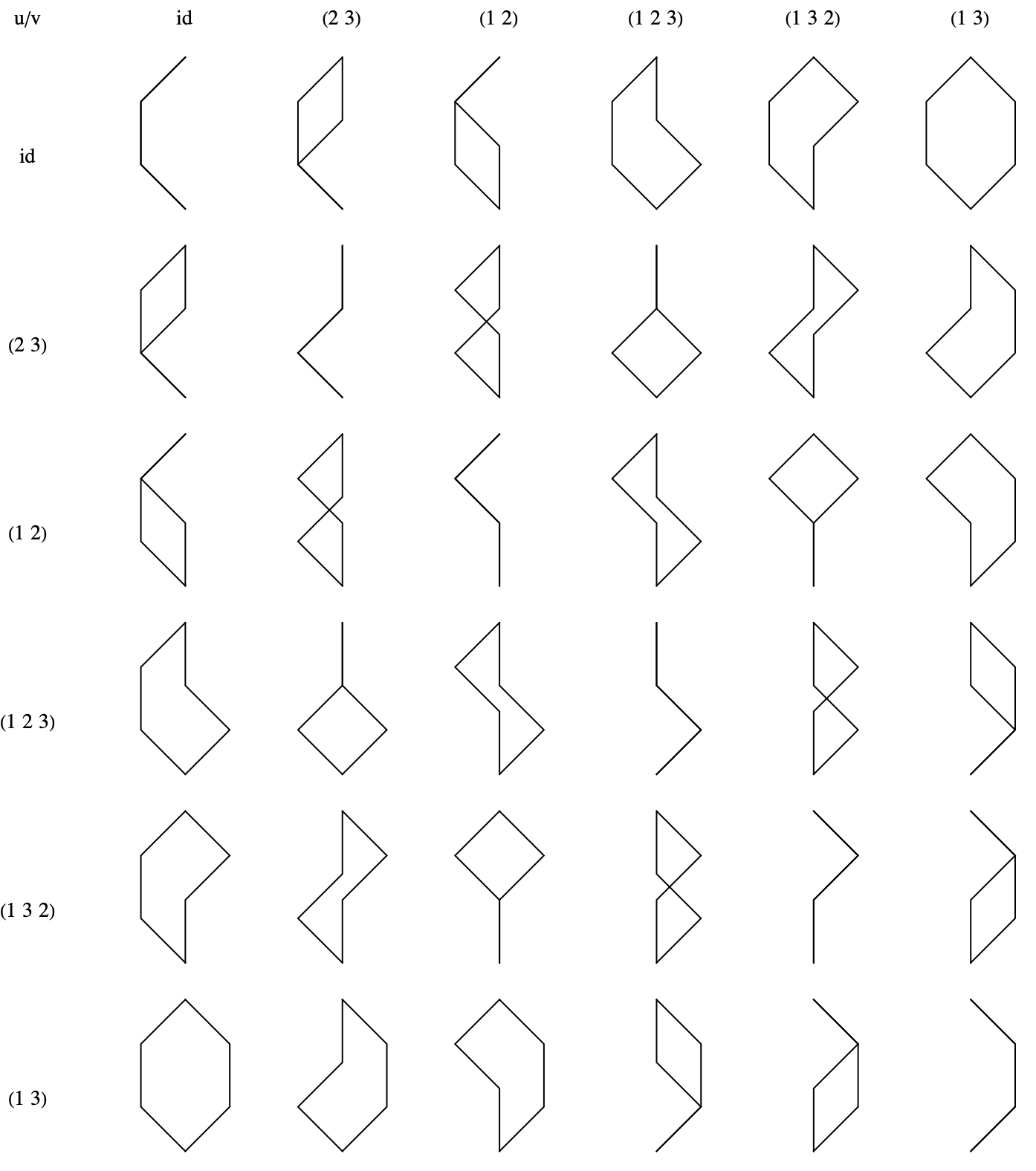}
    \caption{$\mathrm{P}_3^\alpha(u,v)$ for all $u,v \in \Sym{3}$ with $\alpha = \pi/4$.}
    \label{Figure All Etiles Sym3}
\end{figure}

\begin{figure}[H]\vspace{-2cm}
    \hspace{-2cm}
    \includegraphics[width = 18cm, angle =0]{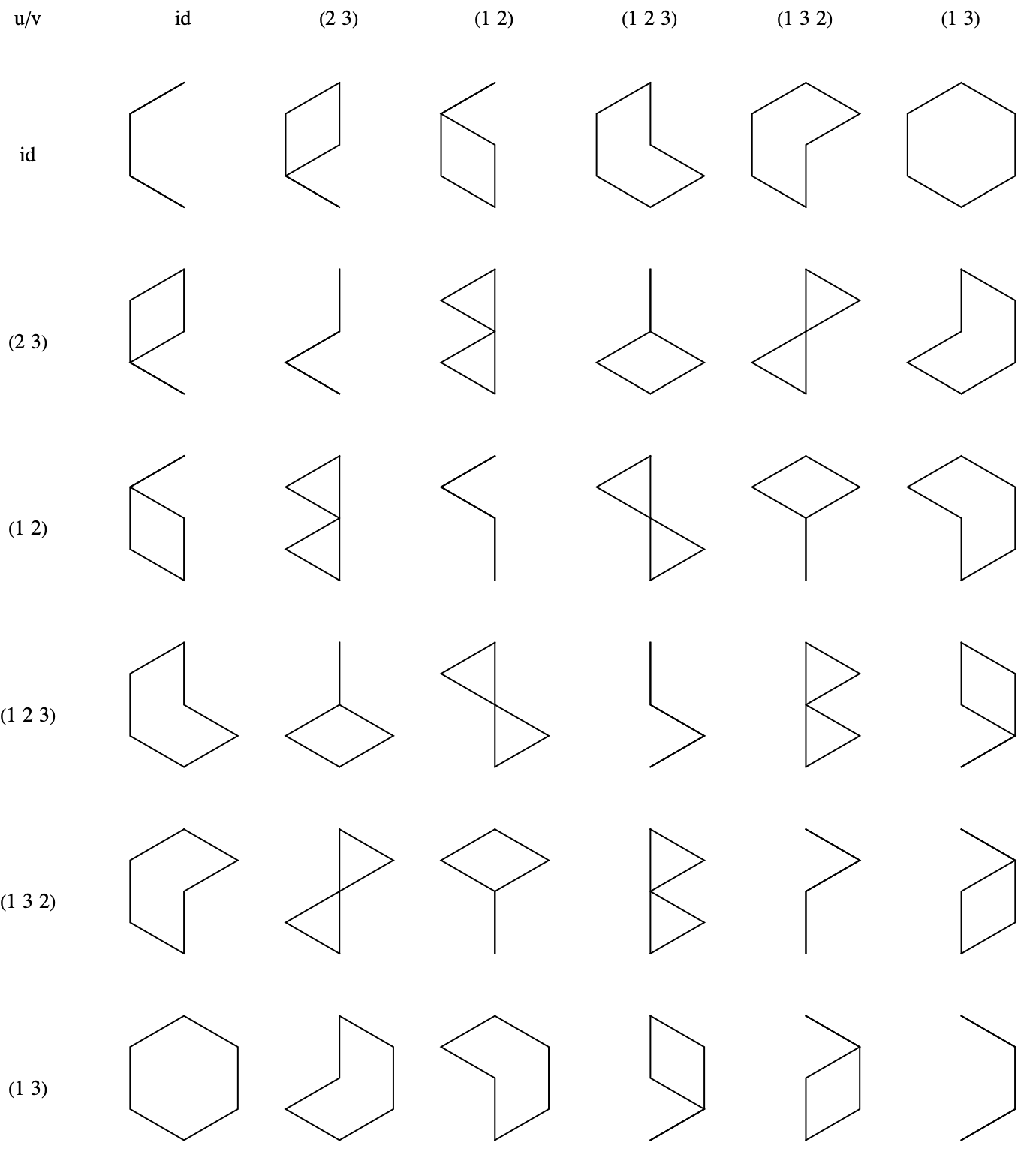}
    \caption{$\mathrm{P}_3^{\alpha}(u,v)$ for all $u,v \in \Sym{3}$ with $\alpha = \pi/6$.}
    \label{Figure Round Etiles S3}
\end{figure}
Note that, by construction, for all $u,v \in \Sym{n}$, $\mathrm{B}^\alpha_n(u) = \mathrm{B}^\alpha_n(v)$ if and only if $u= v$. 

It would be desirable to be able to define a sensible notion of when a pair of borders produce a tile - when does it make sense to do so? We give a crude but general notion of this. Given fixed $n$ and $\alpha$, all borders have the same maximal $y$-coordinate any point may achieve, namely, $h_n^\alpha:=\sum_{k=1}^n \sin\left(\dfrac{(k-1)(\pi - 2\alpha)}{n-1}+\alpha\right)$. For each $0\le y \le h_n^\alpha$, there is a unique point for each border with that $y$-coordinate. Denote the $x$-coordinate of this point by $\mathrm{H}(\mathrm{B}^\alpha_n(w),y)$ for $w \in W$ and $0 \le y \le h_n^\alpha$. 
\begin{definition}\label{precident}
For all $u,v \in W$ we say $\mathrm{B}^\alpha_n(u)$ \textit{precedes} $ \mathrm{B}^\alpha_n(v)$, denoted $\mathrm{B}^\alpha_n(u) \prec \mathrm{B}^\alpha_n(v)$, if for all $0 \le y \le h^\alpha_n$, $$\mathrm{H}(\mathrm{B}^\alpha_n(u),y) \le \mathrm{H}(\mathrm{B}^\alpha_n(v),y).$$
\end{definition}

One can define the interior of any $\mathrm{P}^\alpha_n(u,v)$ to be the union of the set of all line segments whose endpoints are $\mathrm{H}(\mathrm{B}^\alpha_n(u),y)$ and $\mathrm{H}(\mathrm{B}^\alpha_n(v),y)$. However we use the notion of precedence to determine when we assign the word \textit{tile} to some $E$-polygon for reasons that will become apparent in Section \ref{Bijections}.

\begin{definition}\label{tile}
For all $u,v \in W$ we call $\mathrm{P}^\alpha_n(u,v)$ a \textit{tile} if either $\mathrm{B}^\alpha_n(u) \prec \mathrm{B}^\alpha_n(v)$ or $\mathrm{B}^\alpha_n(v) \prec \mathrm{B}^\alpha_n(u)$.
\end{definition}

\begin{figure}[H]
    \centering\hspace{-2.5cm}
\begin{subfigure}[b]{0.1cm}
  \includegraphics[height = 4cm]{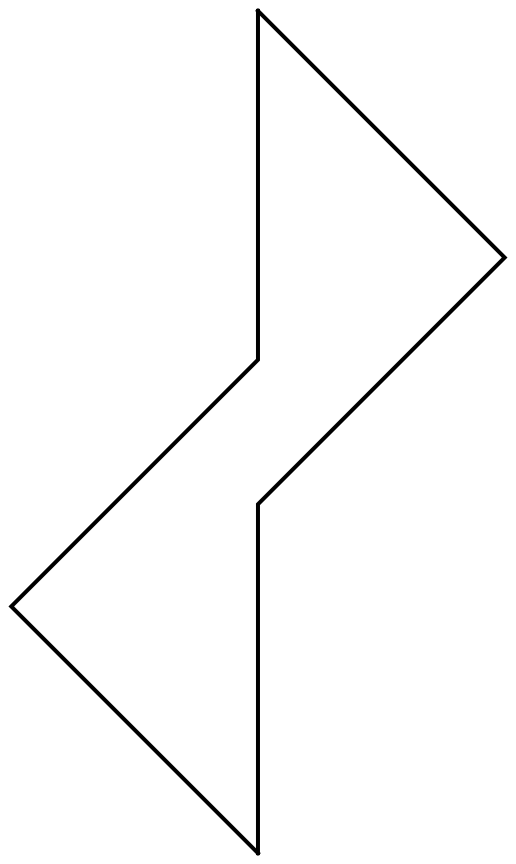}
\end{subfigure}\hspace{5cm}
\begin{subfigure}[b]{0.1cm}
  \includegraphics[height = 4cm]{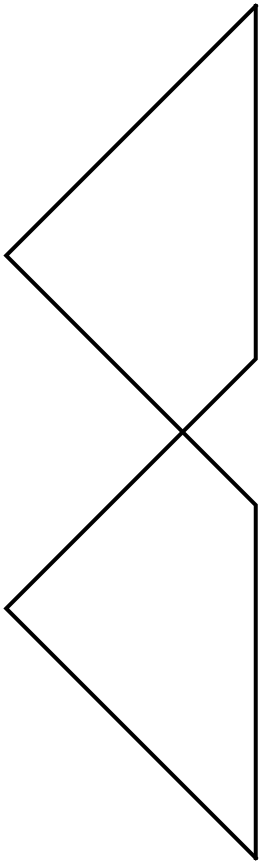}
\end{subfigure}    
\caption{For $W = \Sym{3}$, (left) $\mathrm{P}_3^{\pi/4}((2,3),(1,3,2))$ is a tile and (right) $\mathrm{P}_3^{\pi/4}((2,3),(1,2))$ is not.}
    \label{Figure Intersection Demo}
\end{figure}

We also note here that being a tile is dependent on the choice of $\alpha$ as Figure \ref{FigureAlphaDependance} shows.

\begin{figure}[H]
    \centering\hspace{-2.5cm}
\begin{subfigure}[b]{0.1cm}
  \includegraphics[height = 4cm]{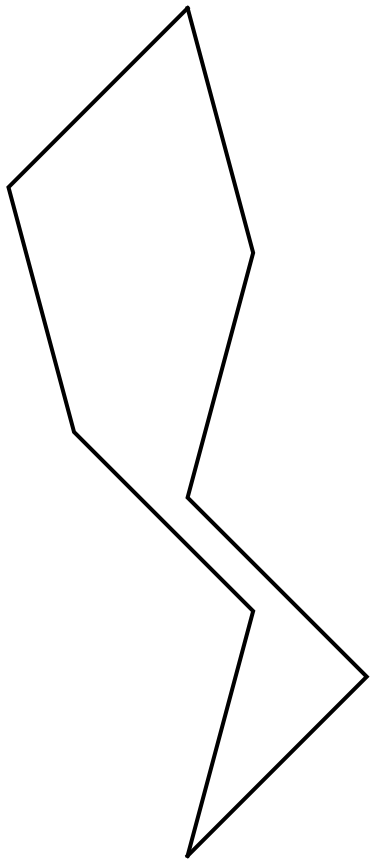}
\end{subfigure}\hspace{5cm}
\begin{subfigure}[b]{0.1cm}
  \includegraphics[height = 4cm]{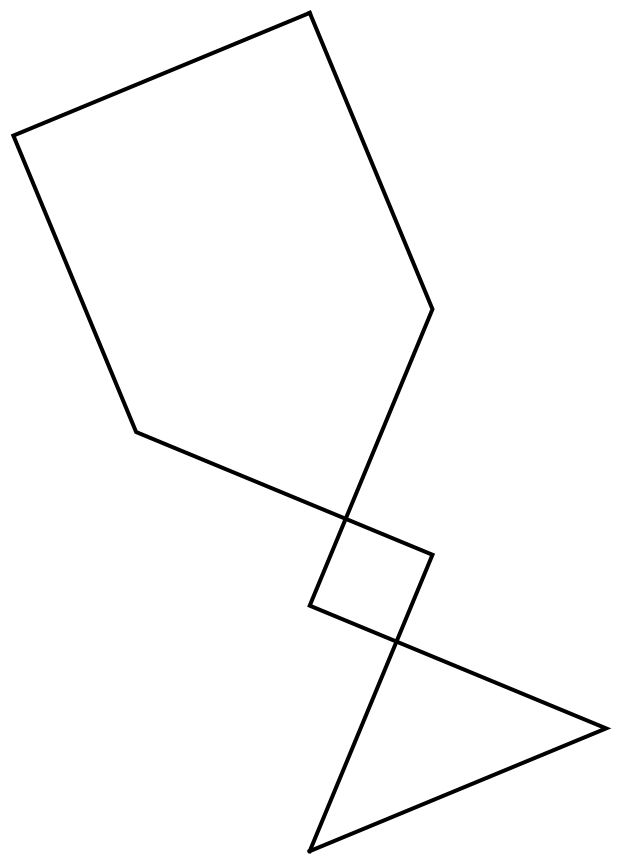}
\end{subfigure}    
\caption{For $W = \Sym{4}$, (left) $\mathrm{P}^{\pi/4}_4((1,2,3),(1,2,4))$ is a tile and (right) $\mathrm{P}^{\pi/8}_4((1,2,3),(1,2,4))$ is not.}
    \label{FigureAlphaDependance}
\end{figure}

\section{Proof of Theorem \ref{symmetric}}\label{Bijection}

The examples presented in the previous section show that self-intersections are dependent on the choice of the minimum angle of edges from the horizontal. Moreover, when this minimum is at least $\pi/4$ for $\Sym{3}$, these self-intersections are in bijection with incomparable elements in the strong Bruhat order. However, in general this is not the case as we observe in Figure \ref{figure criminal non interection} where we have $ W = \Sym{4}$ and elements $(1,2,3)$ and $(1,4,2)$ which are not comparable in the Bruhat order. This phenomenon is the only such pair (up to inverses) amongst the 87 non-comparable elements of $\Sym{4}$ which exhibit this behaviour. 

For what follows, when $\alpha = \pi/4$ we omit $\alpha$ from our notation. 

\begin{figure}[H]
     \centering
      \includegraphics[height=5cm]{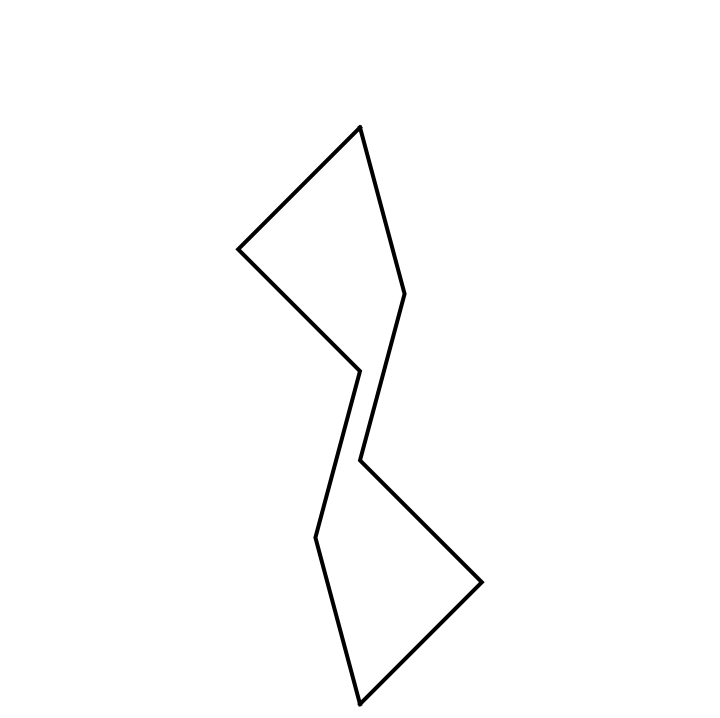}
\caption{For $W = \Sym{4}$, $\mathrm{P}_4((1,3,2),(1,2,4))$ .}
\label{figure criminal non interection}
\end{figure}

We are now ready to prove Theorem~\ref{symmetric} which by way of contrast demonstrates that any two Bruhat comparable elements of $\Sym{n}$ forms an $\mathrm{E}$-polygon which is a tile.  We emphasise that the proof is not dependent on consecutive edges of $\mathcal{B}^\alpha$ being spaced apart by angles of equal measure: they all need only to have absolute angle $\pi/4$ from the horizontal.

\begin{proof}
It's enough to observe this statement for a covering set of relations of the Bruhat order. That is, for all $w \in \Sym{n}$ and $t \in T$, if $w <_B wt$ and $\ell(wt) = \ell(w)+1$, then
$\mathrm{B}_n(w) \prec \mathrm{B}_n(wt)$. But, by Theorem~\ref{BruhatSym}, this is equivalent to the condition that $t = (a,b)$ with $a<b$ and $(a)w^{-1}<(b)w^{-1}$ and no such intermediate $k$ such that $a<k<b$ and $(a)w^{-1}<(k)w^{-1}<(b)w^{-1}$. Let us suppose this is the case. Then the images of $\mathcal{B}_n(w)$ and $\mathcal{B}_n(wt)$ must be identical apart from the transpositions of the vectors, $\beta_n^{(a)w^{-1}}$ and $\beta_n^{(b)w^{-1}}$. Since $a<b$,  $\beta_n^{(a)w^{-1}}$ appears in a lower position to $\beta_n^{(b)w^{-1}}$ as line segment in $\mathrm{B}^\alpha_n(w)$. From $(a)w^{-1}<(b)w^{-1}$, $\beta_n^{(a)w^{-1}}$'s $x$ coordinate is more negative than that of $\beta_n^{(b)w^{-1}}$ - intuitively meaning that $\beta_n^{(a)w^{-1}}$ points further left. Hence, $w <_B wt$ is true implies Figure \ref{borders} is a sufficiently accurate representation of the situation.

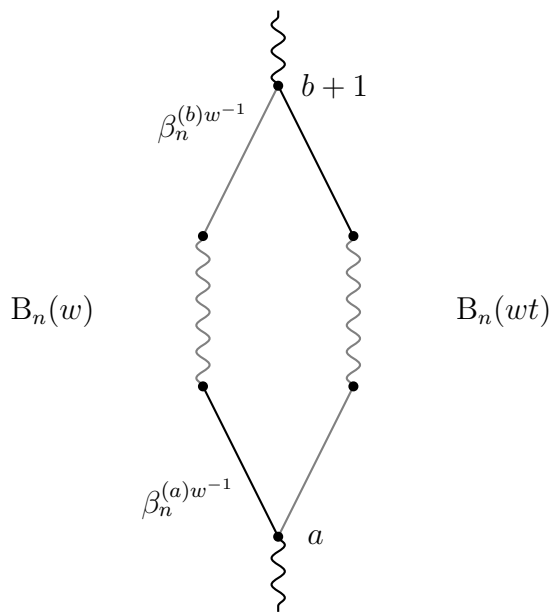
\begin{figure}[H]
    \centering
\begin{tikzpicture}[thick]
  \path [draw=black,snake it]
    (0,8) -- (0,9);
  \path [draw=black,snake it]
    (0,2) -- (0,1);
  \path [draw=gray,snake it]
    (-1,4) -- (-1,6);
  \path [draw=gray,snake it]
    (1,4) -- (1,6);
  \path[draw=black] 
  (0,8) -- (1,6);
   \path[draw=black] 
  (0,2) -- (-1,4);
  \path[draw=gray] 
  (0,8) -- (-1,6);
   \path[draw=gray] 
  (0,2) -- (1,4);
    \filldraw (0,2) circle[radius=1.5pt];
    \filldraw (0,8) circle[radius=1.5pt];
    \filldraw (-1,4) circle[radius=1.5pt];
    \filldraw (-1,6) circle[radius=1.5pt];
    \filldraw (1,4) circle[radius=1.5pt];
    \filldraw (1,6) circle[radius=1.5pt];
    \node (a) at (0.5,2) {$a$};
    \node (b) at (0.75,8) {$b+1$};
    \node (a) at (-1,7.5) {$\beta_n^{(b)w^{-1}}$};
    \node (b) at (-1.2,2.5) {$\beta_n^{(a)w^{-1}}$};
    \node (L) at (-3,5) {$\mathrm{B}_n(w)$};
    \node (R) at (3,5) {$\mathrm{B}_n(wt)$};
\end{tikzpicture}
    \caption{The borders $\mathrm{B}_n(w)$ and $\mathrm{B}_n(wt)$.}
    \label{borders}
\end{figure}

We define \emph{the critical region} to be the union of the edges $\mathrm{B}_n(w)_i$ and $\mathrm{B}_n(wt)_i$ for $\ a\le i\le b$. We will show that $\mathrm{B}_n(w)_i \cap \mathrm{B}_n(wt)_j = \emptyset$ for all $i,j \in \{a,\ldots, b\}$, excluding the common points of $\mathrm{B}_n(w)_a$ and $\mathrm{B}_n(wt)_a$ (the unique point of least $y$-coordinate) and $\mathrm{B}_n(w)_b$ and $\mathrm{B}_n(wt)_b$ (the unique point of largest $y$-coordinate). This is sufficient to prove $\mathrm{B}_n(w) \prec \mathrm{B}_n(w t)$.

Given two distinct vectors $\beta_n^{i},\beta_n^{j}\in \mathcal{B}_n$ with $i<j$, we call the difference between them, $\beta_n^{i} - \beta_n^{j}$, their \emph{difference vector}. We extend this notion to $\mathrm{B}_n(w)$ and $\mathrm{B}_n(wt)$ by defining the difference vector of these borders to be the difference vector of $\beta_n^{(a)w^{-1}}$ and $\beta_n^{(b)w^{-1}}$. Note that this difference vector is equal to the difference of $\Sigma_{j=1}^{c}\mathcal{B}^\alpha_n(w)_j$ and $\Sigma_{j=1}^{c}\mathcal{B}^\alpha_n(wt)_j$ for all  $a\le c < b$ respectively. 

We examine some of the properties our difference vectors may possess. Consider two distinct vectors in this region and let $\gamma$ and $\theta$ denote their angles from $\begin{pmatrix}
1\\
0
\end{pmatrix}$ measured in an anticlockwise rotation, with $\gamma<\theta$ say. By construction, the angles for each $\beta^i_n$ can possibly take is within the range $(\frac{\pi}{4},\frac{3\pi}{4})$. The gradient of the chord is the same as the tangent to the circle at the point that intersects the bisector of the chord. The bisector is that vector with angle $\frac{\gamma+\theta}{2}$ and hence the tangent has angle $\frac{\gamma+\theta}{2}-\frac{\pi}{2}$. So the range of gradients a difference vector can take is contained in the open interval $(-\frac{\pi}{4},\frac{\pi}{4})$.
\begin{figure}[H]
    \centering
    \begin{tikzpicture}[scale=0.8]
    \draw[black,thick] (5,0) arc (360:0:5);  
    \draw[gray,dashed] (-5,0) -- (5,0);  
    \draw[gray,dashed] (0,0) -- (3.5355,3.5355);  
    \draw[gray,dashed] (0,0) -- (-3.5355,3.5355); 
    \draw[scale=1,black,thick,-Latex] (0,0) -- (3,4);  
    \draw[scale=1,black,thick,-Latex] (0,0) -- (-1.913417,4.619397);
    \draw[blue,thick,-Latex]   (-1.913417,4.619397) -> (3,4);
    \filldraw (-1.913417,4.619397)  circle[black,radius=1.5pt];
    \filldraw (3,4) circle[black,radius=1.5pt];
    \filldraw[black,thick] (0,0)--(1,0) arc (0:45:1)--cycle; 
    \node (a) at (1.25,0.5) {\Large $\frac{\pi}{4}$};
    \filldraw[black,thick] (0,0)--(-1,0) arc (0:-45:-1)--cycle; 
    \node (a) at (-1.25,0.5) {\Large $\frac{\pi}{4}$};
    \end{tikzpicture}
    \caption{The difference chord between two vectors.}
    \label{chords}
\end{figure}
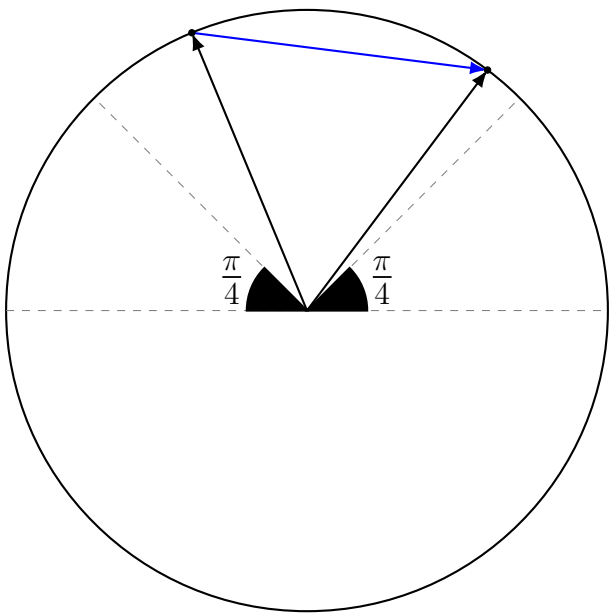

Suppose we do have an non-empty intersection of $\mathrm{B}_n(w)_i$ and $\mathrm{B}_n(wt)_j$. Without loss of generality, we may assume $a \le i \le b$. We consider the three cases of $|i-j|=0$, $|i-j|=1$ and $|i-j|\ge 2$ separately.

If $|i-j|=0$, then $\mathrm{B}_n(w t)_i = \mathrm{B}_n(w)_i$ and hence they are non-equal, parallel edges and so do not intersect. If $i = a$ or $i = b$ then the vectors only intersect in their common vertices. All other vectors are equal and not in the critical region.

Next, we now show that if $|i-j|=1$ we still have no intersections in the critical region. Suppose, without loss of generality, that $\mathrm{B}_n(w)_i$ intersects $\mathrm{B}_n(w t)_{i+1}$ giving us the scenario described in Figure \ref{Figure Intersection Rhombus}:

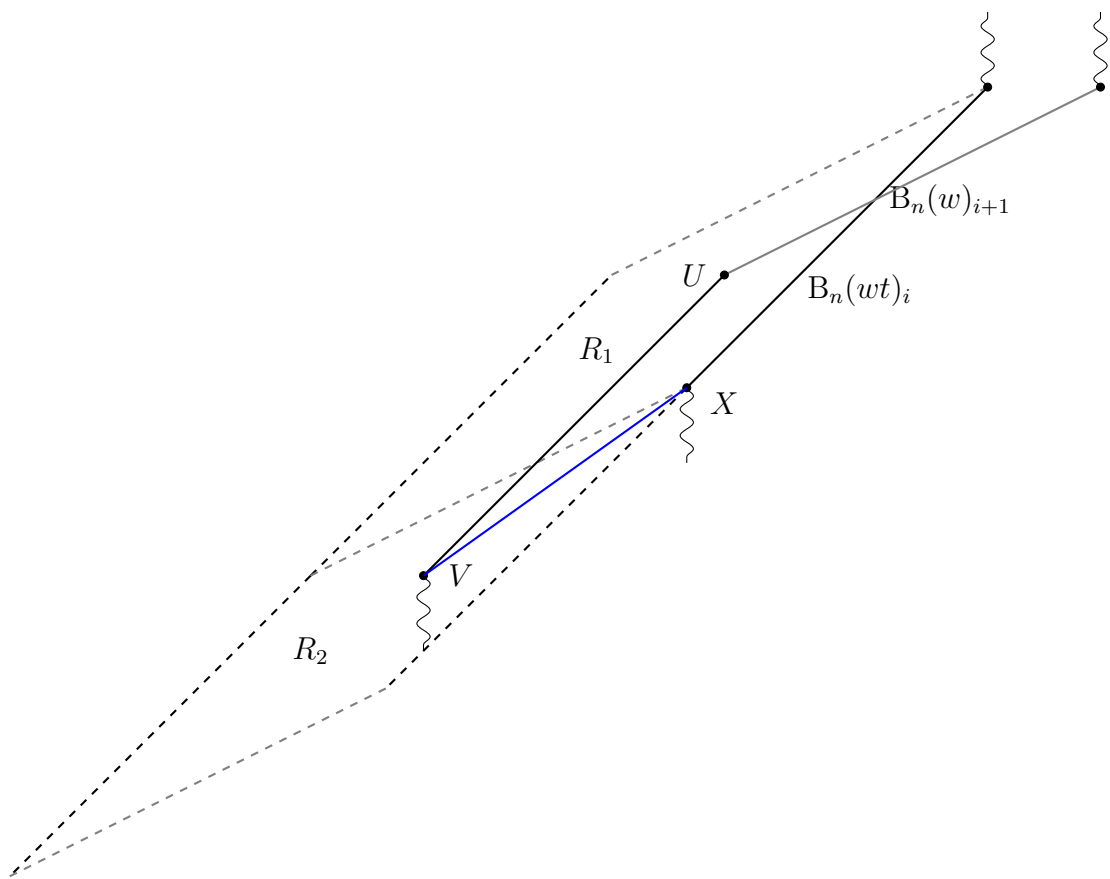
\begin{figure}[H]
    \hspace{-1cm}
    \begin{tikzpicture}[scale=1.0]
      \draw[black,thick,dashed] (0,0) -- (4,4);
            \draw[black,thick,dashed] (0,0) -- (-4,-4);
            \draw[black,thick,dashed] (5,2.5) -- (1,-1.5);
            \draw[gray,thick,dashed] (-4,-4) -- (1,-1.5);
       \node[black] at (7.3,3.8) {$\mathrm{B}_n(w t)_{i}$};
      \draw[gray,thick,dashed] (0,0) -- (5,2.5);
       \node[black] at (8.5,5) {$\mathrm{B}_n(w)_{i+1}$};
      \draw[black,thick] (5,2.5) -- (9,6.5);
       \draw[gray,thick,dashed] (4,4) -- (9,6.5);
      \filldraw (5,2.5) circle[black,radius=1.5pt];  
\filldraw (9,6.5) circle[black,radius=1.5pt]; 
    \path [draw=black,snake it]
    (5,2.5) -- (5,1.5);
    \path [draw=black,snake it]
    (10.5,6.5) -- (10.5,7.5);
      \draw[gray,thick] (5.5,4) -- (10.5,6.5);
    \draw[black,thick] (5.5,4) -- (1.5,0);
      \filldraw (5.5,4) circle[black,radius=1.5pt];  
    \filldraw (1.5,0) circle[black,radius=1.5pt];
\filldraw (10.5,6.5) circle[black,radius=1.5pt]; 
    \path [draw=black,snake it]
    (1.5,0) -- (1.5,-1);
    \path [draw=black,snake it]
    (9,6.5) -- (9,7.5);
    \draw[blue,thick] (1.5,0) -- (5,2.5); 
    \node[black] at (2,0) {$V$};
    \node[black] at (5.5,2.3) {$X$};
    \node[black] at (5.1,4) {$U$};
    \node[black] at (3.8,3) {$R_1$};
    \node[black] at (0,-1) {$R_2$};
 \end{tikzpicture}
    \caption{The rhombic region labelled $R_2$ for which the lower vertex of $\mathrm{B}_n(w)_{i}$ lies in if and only if  $\mathrm{B}_n(w)_{i+1}$ and $\mathrm{B}_n(w t)_{i}$ intersect.}
    \label{Figure Intersection Rhombus}
\end{figure}

Let $U$ and $V$ be the upper and lower vertices of $\mathrm{B}_n(w)_{i+1}$ respectively and $X$ be the lower vertex of $\mathrm{B}_n(wt)_{i}$. Note that $U$ is in the rhombic region labelled $R_1$ if and only if $\mathrm{B}_n(w)_{i+1}$ and $\mathrm{B}_n(wt)_{i}$ intersect. Equivalently, this is true exactly when $V$ is in the  region labelled $R_2$. Observe that the line segment from $V$ to $X$ is equivalent to the difference vector of $\mathrm{B}_n(w)$ and $\mathrm{B}_n(wt)$. But then $V$ is in $R_2$ if and only if the difference vector has angle from $\begin{pmatrix}
1\\
0
\end{pmatrix}$ strictly between $\beta_n^{a^{-1}(w)}$ and  $\beta_n^{b^{-1}(w)}$. But this is a contradiction as we saw the angles of $\beta_n^{a^{-1}(w)}$ and $\beta_n^{b^{-1}(w)}$ lie in $[\pi/4, 3\pi/4]$ whereas the angles of difference vectors lie in the disjoint, open interval $(-\pi/4,\pi/4)$.

For $|i-j|\ge 2$ we first consider $|i-j|= 2$ where the situation pictured in Figure \ref{Figure Rhombic Heights} applies.

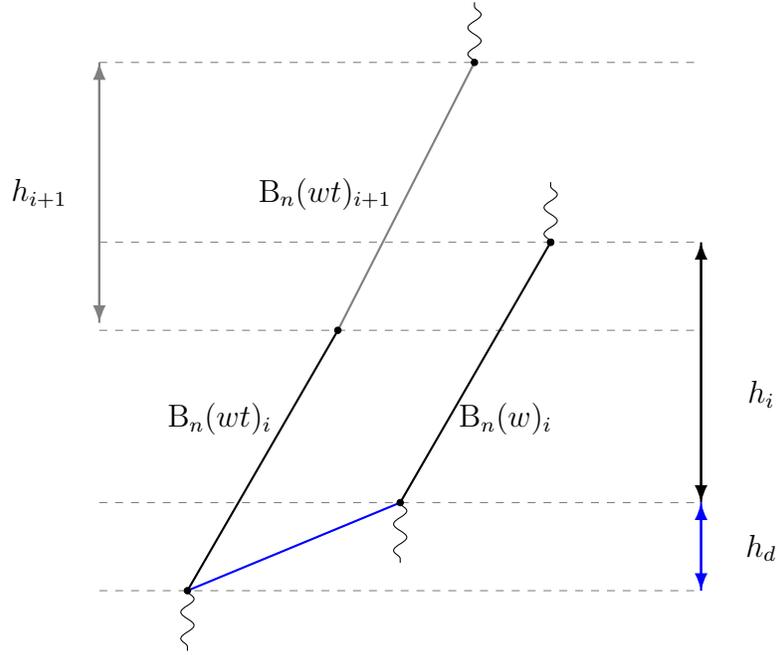
\begin{figure}[H]
    \centering
    \begin{tikzpicture}[scale=4]
    \draw[scale=1,gray,dashed] (-1,0.70710678118) -- (1,0.70710678118);
    \draw[scale=1,gray,dashed] (-1,0.70710678118+0.86602540378) -- (1,0.70710678118+0.86602540378);
    \draw[scale=1,gray,dashed] (-1,1) -- (1,1);
    \draw[scale=1,gray,dashed] (-1,0.70710678118+0.86602540378+
    0.89100652418) -- (1,0.70710678118+0.86602540378+
    0.89100652418);
    \draw[scale=1,gray,dashed] (-1,1+0.86602540378) -- (1,1+0.86602540378);

    \draw[scale=1,black,-Latex,thick] (1,1) -> (1,1+0.86602540378);
    \draw[scale=1,black,-Latex,thick]  (1,1+0.86602540378)->(1,1);
    \node[black] at (1.2,1.36) {$h_i$};

    \draw[scale=1,blue,-Latex,thick] (1,0.70710678118) -> (1,1);
    \draw[scale=1,blue,-Latex,thick]  (1,1)->(1,0.70710678118) ;
    \node[black] at (1.2,0.85) {$h_d$};
    
    \draw[scale=1,gray,-Latex,thick]
    (-1,+0.70710678118+
    0.89100652418) -> (-1,0.70710678118+0.86602540378+
    0.89100652418);
    \draw[scale=1,gray,-Latex,thick]  (-1,0.70710678118+0.86602540378+
    0.89100652418)->(-1,0.707106781188+
    0.89100652418) ;
    \node[black] at (-1.2,0.43+0.707106781188+
    0.89100652418) {$h_{i+1}$};
    
    \node[black] at (-0.25,0.43+0.707106781188+
    0.89100652418) {$\mathrm{B}_n(wt)_{i+1}$};
    \node[black] at (-0.6,1.28) {$\mathrm{B}_n(wt)_{i}$};
    \node[black] at (0.35,1.28) {$\mathrm{B}_n(w)_{i}$};
    
    \draw[scale=1,blue,thick] (0,1) -- (-0.70710678118,0.70710678118);  
    \draw[scale=1,gray,thick] (-0.70710678118+0.5,0.70710678118+0.86602540378) -- (-0.70710678118+0.5+0.45399049974,0.70710678118+0.86602540378+
    0.89100652418);  
    \draw[scale=1,black,thick] (0+0.5,1+0.86602540378) -- (0,1);  
    \draw[scale=1,black,thick] (-0.70710678118+0.5,0.70710678118+0.86602540378) -- (-0.70710678118,0.70710678118);
    \path [draw=black,snake it]
    (-0.70710678118+0.5+0.45399049974,0.70710678118+0.86602540378+
    0.89100652418)  -- (-0.70710678118+0.5+0.45399049974,0.70710678118+0.86602540378+
    0.89100652418+0.2);  
    \path [draw=black,snake it]
    (0,1) -- (0,0.8);
    \path [draw=black,snake it]
    (-0.70710678118,0.70710678118) -- (-0.70710678118,0.70710678118-0.2);
    \filldraw (-0.70710678118,0.70710678118) circle[black,radius=0.3pt];
    \path [draw=black,snake it]
    (0+0.5,1+0.86602540378) -- (0+0.5,1+0.86602540378+0.2);
    \filldraw (0,1) circle[black,radius=0.3pt];
    \filldraw (0+0.5,1+0.86602540378) circle[black,radius=0.3pt];
    \filldraw (0+0.5,1+0.86602540378) circle[black,radius=0.3pt];
    \filldraw (-0.70710678118+0.5,0.70710678118+0.86602540378) circle[black,radius=0.3pt];
    \filldraw
    (-0.70710678118+0.5+0.45399049974,0.70710678118+0.86602540378+
    0.89100652418)circle[black,radius=0.3pt]; 
    \end{tikzpicture}
    \caption{The heights concerning the $|i-j|\ge 2$ case.}
    \label{Figure Rhombic Heights}
\end{figure}

Note that $\mathrm{B}_n(wt)_i$ is a translation of $\mathrm{B}_n(w)_i$ since $i$ is in the critical strip. If the $y$-coordinate of $\mathrm{B}_n(w)_i$ is less than that of $\mathrm{B}_n(wt)_{i+2}$, then an intersection between  $\mathrm{B}_n(wt)_{i+2}$ and $\mathrm{B}_n(w)_i$ is impossible. So we consider the inequality $h_d+h_i < h_i+h_{i+1}$, or equivalently, $h_d < h_{i+1}$. Note that $h_d$ is bounded above by  $1-\dfrac{\sqrt{2}}{2}$ and $h_{i+1}$ is bounded below by $\dfrac{\sqrt{2}}{2}$. But $1-\dfrac{\sqrt{2}}{2} < \dfrac{\sqrt{2}}{2}$ and so $h_d \le 1-\dfrac{\sqrt{2}}{2} < \dfrac{\sqrt{2}}{2} \le h_{i+1}$. Therefore, $\mathrm{B}_n(w)_i$ and $\mathrm{B}_n(wt)_{i+1}$ certainly do not intersect.
But since the $y$-coordinate is strictly increasing in borders we know that no intersection can occur for all $|i-j|\ge 2$ also.

\end{proof}

Taking $\alpha = \pi/4$ ensures that for the case $|i-j| = 1$, $\mathrm{B}_n(w)_i$ and $\mathrm{B}_n(wt)_j$ have an empty intersection. However if $\alpha < \pi/4$, a non-trivial intersection can occur for some sufficiently large $n$. So $\alpha = \pi/4$ is sharp precisely in this sense.  When $|i-j|\ge 2$, the least $\alpha$ needed to ensure that $\mathrm{B}_n(w)_i$ and $\mathrm{B}_n(wt)_j$ have an empty intersection in the critical region is $\alpha = \pi/6$.

An important corollary evident from the proof of Theorem \ref{symmetric} is that $\mathrm{B}_n(w) \cap \mathrm{B}_n(wt)$ differ on exactly those points on edges in the so-called critical region. Furthermore, in that region, the points on $\mathrm{B}_n(wt)$ have an $x$-coordinate greater than their $\mathrm{B}_n(w)$ counterparts of the same $y$-coordinate. Putting this more precisely, we have Corollary \ref{sym col}.

\begin{corollary}\label{sym col}
Suppose $w,t\in \mathrm{Sym}(n)$ with $t \in T$ and $w <_B wt$ and $\ell(wt) = \ell(w)+1$. Write $t = (a,b)$ for some $1\le a < b \le n$. Then $\mathrm{H}(\mathrm{B}_n(w),y) < \mathrm{H}(\mathrm{B}_n(wt),y)$ for all $y$ satisfying $$\sum_{k=1}^{a-1}\sin\left(\dfrac{((k-1)w)(\pi - 2\alpha)}{n-1}+\alpha\right)<y<\sum_{k=1}^{b}\sin\left(\dfrac{((k-1)w)(\pi - 2\alpha)}{n-1}+\alpha\right),$$ (declaring $\sum_{k=1}^{a-1}\sin\left(\dfrac{((k-1)w)(\pi - 2\alpha)}{n-1}+\alpha\right) = 0$ for when $a=1$), and $\mathrm{H}(\mathrm{B}_n(w),y) = \mathrm{H}(\mathrm{B}_n(wt),y)$ otherwise.
\end{corollary}
\section{Tiling Bijections}\label{Bijections}
In this closing section we extend the definitions of Section
\ref{section E-polygons} to general E-embeddings and prove the consequent bijections.
\begin{definition}
Let $\varphi: W \hookrightarrow \mathrm{Sym}(n)$ be an  E-embedding. Then we define $\mathrm{B}^\alpha_\varphi(w)$ to be $\mathrm{B}^\alpha_n(\varphi(w))$ and we set $\mathrm{P}^\alpha_\varphi(u,v)$ to be $\mathrm{P}^\alpha(\varphi(u),\varphi(v))$ for all $u,v,w \in W$.
\end{definition}

We define the \textit{tiling} of a reduced expression.

\begin{definition}
Let $\varphi: W \hookrightarrow \mathrm{Sym}(n)$ be an E-embedding. Let $s_{i_1}s_{i_2}\ldots s_{i_k}$ be a reduced expression of some $w \in W$. Then we define the \textit{tiling} of $s_{i_1}s_{i_2}\ldots s_{i_k}$ to be $$\mathrm{T}_\varphi(s_{i_1}s_{i_2}\ldots s_{i_k}) = \bigcup_{j=0}^k \mathrm{B}^\alpha_\varphi(s_{i_1}\ldots s_{i_j})$$ where it is agreed that when $j=0$,  $\mathrm{B}^\alpha_\varphi(s_{i_1}\ldots s_{i_j}) = \mathrm{B}^\alpha_\varphi(id)$. 
\end{definition}

We now prove Theorem \ref{lemma E-embedding}. 

\begin{proof}
Suppose $w = s_{i_1}\ldots s_{i_k}$ is a reduced word for some element $w \in W$. Necessarily, we have 
\begin{align*}
    &\mathrm{id} <_R s_{i_1} <_R \ldots <_R s_{i_1}\ldots s_{i_k}. \\ \intertext{Since $\varphi$ is an $E$-embedding, we know }
    &\varphi(\mathrm{id}) <_B \varphi(s_{i_1}) <_B \ldots <_B \varphi(s_{i_1}\ldots s_{i_k}) \\ \intertext{and then Theorem \ref{symmetric} implies }
    &\mathrm{B}_\varphi(\mathrm{id}) \prec \mathrm{B}_\varphi(s_{i_1}) \prec \ldots \prec \mathrm{B}_\varphi(s_{i_1}\ldots s_{i_k}). 
\end{align*}
\end{proof}

The significance of this theorem is that what we call tilings are indeed deserving of their name. Since $\mathrm{B}_\varphi(\mathrm{id}) \prec \mathrm{B}_\varphi(s_{i_1}) \prec \ldots \prec \mathrm{B}_\varphi(s_{i_1}\ldots s_{i_k}) $ for each reduced word, each of $\mathrm{P}_\varphi(\mathrm{id},s_{i_1}), \mathrm{P}_\varphi(s_{i_1},s_{i_1}s_{i_2}) , \ldots , \mathrm{P}_\varphi(s_{i_1}\ldots s_{i_{k-1}},s_{i_1}\ldots s_{i_k})$ forms a tile and set of the interiors of these tiles partition the interior of the polygon $\mathrm{P}_\varphi(s_{i_1}\ldots s_{i_k})$.

\begin{definition}
Let $\varphi: W \hookrightarrow \mathrm{Sym}(n)$ be an E-embedding and $w \in W$. Let $\mathcal{T}^\alpha_\varphi(w)$ be the set consisting of all $\mathrm{T}_\varphi(s_{i_1}s_{i_2}\ldots s_{i_k})$ for all reduced words $s_{i_1}s_{i_2}\ldots s_{i_k}$ evaluating to $w$.
\end{definition}

\begin{definition}
Let $\varphi: W \hookrightarrow \mathrm{Sym}(n)$ be an  E-embedding. Let $r,s \in S$, then we define $J_{\varphi}$ so that $\{r,s\} \in J_{\varphi}$ if and only if for all reduced words containing a consecutive subword of $\underbrace{srs\ldots }_{m_{s,r}}$, we have
$$\mathrm{T}_\varphi(s_{i_1}\ldots s_{i_\ell }\underbrace{srs\ldots }_{m_{s,r}} s_{i_r}\ldots s_{i_j}) = \mathrm{T}_\varphi(s_{i_1}\ldots s_{i_\ell }\underbrace{rsr\ldots }_{m_{s,r}} s_{i_r}\ldots s_{i_j}).$$
\end{definition}

Recall that if $\varphi(s) = \prod_{i=1}^k(a_i,b_i)$ with $a_i<b_i$ for each $i= 1,\ldots, k$, then we define $\mathrm{I}_\varphi(s) = \bigcup_{i=1}^{k}\{a_i,a_i+1\ldots,b_i\}$. Necessarily, there must exists a unique  choice of subsets $\{a'_1,\ldots,a'_{k'} \} \subseteq \{a_1,\ldots,a_k\}$,  and $\{b'_1,\ldots,b'_{k'}\} \subseteq \{b_1,\ldots,b_k\}$ such that $$\mathrm{I}_\varphi(s) = \bigcup_{m=1}^{k'}\{a'_m,a'_m+1,\ldots,b'_m\}$$ and that the intervals in this union are pairwise disjoint. We call this the disjoint form of $\mathrm{I}_\varphi(s)$ and $a'_i$ and $b'_i$ its $i^{th}$ disjoint representatives.
Given this notation, we can provide a natural extension to Corollary \ref{sym col} in the setting of E-embeddings. This shows that support intervals determine exactly the points at which the borders for $w$ and $ws$ differ.

\begin{lemma}\label{col reps}
Let $\varphi : W \hookrightarrow \mathrm{Sym}(n)$ be an E-embedding. Let $w \in W$ and $s \in S$ be such that $w <_R ws$. Suppose $\mathrm{I}_\varphi(s) = \bigcup_{m=1}^{k'}\{a'_m,a'_m+1,\ldots,b'_m\}$ in disjoint form. Then $\mathrm{H}(\mathrm{B}_\varphi(w),y) < \mathrm{H}(\mathrm{B}_\varphi(ws),y)$ if and only if for some $m\in\{1,\ldots,k'\}$, $$\sum_{i=1}^{a'_m-1}\sin\left(\dfrac{((i-1)\varphi(w))(\pi - 2\alpha)}{n-1}+\alpha\right)<y<\sum_{i=1}^{b'_m}\sin\left(\dfrac{((i-1)\varphi(w))(\pi - 2\alpha)}{n-1}+\alpha\right)$$  and  $\mathrm{H}(\mathrm{B}_\varphi(w),y) = \mathrm{H}(\mathrm{B}_\varphi(ws),y)$ otherwise. 
\end{lemma}
\begin{proof}
Since $w<_R ws$ and $\varphi$ is an E-embedding, $\varphi(w) <_B \varphi(ws)$. Therefore, there exists a sequence of transpositions in $\mathrm{Sym}(n)$, $t_1,\ldots, t_\ell$, such that 
$$\varphi(w) <_B \varphi(w)t_1 <_B \ldots <_B \varphi(w)t_1\ldots t_{\ell-1} <_B \varphi(w)t_1\ldots t_\ell= \varphi(ws)$$ and $\ell(\varphi(w)t_1\ldots t_{j}) =\ell(\varphi(w)) +j.$  

Consider $\varphi(s)$ restricted to a given interval, $\{a'_{m},\ldots,b'_{m}\}$, and call this induced permutation $\varphi_{|_m}(s)$. This is an element of the symmetric group $\mathrm{Sym}(\{a'_{m},\ldots,b'_{m}\})$, itself a parabolic subgroup of $\mathrm{Sym}(n) = \mathrm{Sym}(\{1,\ldots,n\})$. Using Corollaries 1.4.4 and 1.4.8 of \cite{BB}, we know that any reduced expression of $\varphi_{|_m}(s)$ is the product of adjacent transpositions using only those in $\mathrm{Sym}(\{a'_{m},\ldots,b'_{m}\})$. Furthermore, to write $\varphi_{|_m}(s)$ as the minimal product of (not necessarily adjacent) transpositions, $\varphi_{|_m}(s) = t'_1,\ldots t'_{\ell'}$ with $\ell(t'_1,\ldots t'_{i}) = i $, implies each $t'_i$ is in $\mathrm{Sym}(\{a'_{m},\ldots,b'_{m}\})$ also. For each $m$, we can repeatedly apply Corollary \ref{sym col} to each $t'_i$ for $i=1,\ldots,\ell'$ whence the result.
\end{proof}

We now prove Lemma \ref{Jdescription}.
\begin{proof}
Let  $s,r \in S$ and take some reduced word containing $\underbrace{rsr\ldots }_{m_{s,r}}$ as a consecutive subword, $w = s_{i_1}\ldots s_{i_\ell }\underbrace{rsr\ldots }_{m_{s,r}} s_{i_r}\ldots s_{i_j}$. Since $\varphi$ is an embedding, $s$ and $r$ are distinct if and only if $\varphi(s)$ and $\varphi(r)$ are distinct also. Moreover, we must have $\mathrm{B}_\varphi(us) \ne \mathrm{B}_\varphi(ur)$ for all $u \in W$. 

Write $v  = s_{i_1}\ldots s_{i_\ell}$. Lemma \ref{col reps} shows that support intervals determine exactly the points at which $\mathrm{B}_\varphi(v s))$ and $\mathrm{B}_\varphi(v r))$ differ from $\mathrm{B}_\varphi(v ))$ respectively. Let 
\begin{align*}
    \mathrm{I}_\varphi(s) &= \bigcup_{m=1}^{k'}\{a'_m,a'_m+1,\ldots,b'_m\}, \\
    \mathrm{I}_\varphi(r) &= \bigcup_{j=1}^{h'}\{c'_j,c'_j+1,\ldots,d'_j\}
\end{align*}
be the disjoint forms of $\mathrm{I}_\varphi(s)$ and $\mathrm{I}_\varphi(r)$ respectively. 

Suppose $\mathrm{I}_\varphi(s) \cap \mathrm{I}_\varphi(r) \ne \emptyset$. First we consider the case that $\{a'_m,a'_m+1,\ldots,b'_m\} = \{c'_j,c'_j+1,\ldots,d'_j\}$ for some $m$ and $j$, and that the restricted permutations, $\varphi(s)_{|_m}$ and  $\varphi(r)_{|_j}$, are equal. Then $\varphi(v rs)_{|_m} =\varphi(v )_{|_m}$.  Now consider those edges of $\mathrm{B}_\varphi(v ), \mathrm{B}_\varphi(v s)$ and  $\mathrm{B}_\varphi(v sr)$ indexed by $\{a'_m,a'_m+1,\ldots,b'_m\}$: they are identical for $\mathrm{B}_\varphi(v )$ and $\mathrm{B}_\varphi(v sr)$. Hence it cannot be true that $\mathrm{B}_\varphi(v ) \prec\mathrm{B}_\varphi(v s) \prec \mathrm{B}_\varphi(v sr)$, contradicting Theorem \ref{lemma E-embedding} since  $v  <_R v s <_R v sr$.

Now suppose there exists some $j$ and $m$ such that $\{a'_m,a'_m+1,\ldots,b'_m\} \cap \{c'_j,c'_j+1,\ldots,d'_j\} \ne \emptyset$ and that the restricted permutations, $\varphi(s)_{|_m}$ and $\varphi(r)_{|_j}$, are not equal (they may not even be restricted to the same set). Since they are not equal, we must have points at which they differ. More specifically, Lemma \ref{col reps} shows there must exist some $$\sum_{i=1}^{\min(a'_m,c'_j)-1}\sin\left(\dfrac{((i-1)\varphi(v))(\pi - 2\alpha)}{n-1}+\alpha\right)<y<\sum_{i=1}^{\min(b'_m,d'_j)}\sin\left(\dfrac{((i-1)\varphi(v))(\pi - 2\alpha)}{n-1}+\alpha\right)$$ where, without loss of generality, $\mathrm{H}(\mathrm{B}_\varphi(v ),y)<\mathrm{H}(\mathrm{B}_\varphi(v s),y)<\mathrm{H}(\mathrm{B}_\varphi(v r),y)$. We know that $v <_Rv s$ and $v <_Rv r$. Now consider the consequences if $$\mathrm{T}(s_{i_1}\ldots s_{i_\ell }\underbrace{rsr\ldots }_{m_{s,r}} s_{i_r}\ldots s_{i_j}) = \mathrm{T}(s_{i_1}\ldots s_{i_\ell }\underbrace{srs\ldots }_{m_{s,r}} s_{i_r}\ldots s_{i_j}).$$ There must be some prefix of the word $s_{i_1}\ldots s_{i_\ell }\underbrace{rsr\ldots }_{m_{s,r}} s_{i_r}\ldots s_{i_j}$, say $w_j := s_{i_1}\ldots s_{i_j}$, such that $\mathrm{H}(\mathrm{B}_\varphi(w_j)),y) = \mathrm{H}(\mathrm{B}_\varphi(v s),y)$. Necessarily, $j\ne \ell,\ell+1$. If $j< \ell$, then this contradicts Theorem \ref{lemma E-embedding} as $w_j <_R v $ but $$\mathrm{H}(\mathrm{B}_\varphi(v ),y)<\mathrm{H}(\mathrm{B}_\varphi(v s),y) = \mathrm{H}(\mathrm{B}_\varphi(w_j)),y).$$ If $j> \ell+1$, then this contradicts Theorem \ref{lemma E-embedding} again as $v r <_R w_j$ but $$\mathrm{H}(\mathrm{B}_\varphi(w_j),y) = \mathrm{H}(\mathrm{B}_\varphi(v s),y)<\mathrm{H}(\mathrm{B}_\varphi(v r),y).$$

So far, we have proved if $\{s,r\} \in \mathrm{J}_\varphi$, then $\mathrm{I}_\varphi(s) \cap \mathrm{I}_\varphi(r) \ne \emptyset$. To see that the converse is true, suppose $\mathrm{I}_\varphi(s) \cap \mathrm{I}_\varphi(r) = \emptyset$. Lemma \ref{col reps} shows that for all $y$, only one of  $\mathrm{H}(\mathrm{B}_\varphi(v s),y)$ and $\mathrm{H}(\mathrm{B}_\varphi(v r),y)$ can differ from $\mathrm{H}(\mathrm{B}_\varphi(v ),y)$. Moreover since $\mathrm{I}_\varphi(s) \cap \mathrm{I}_\varphi(r) = \emptyset$, $s$ and $r$ must commute. Hence $\mathrm{B}_\varphi(v sr) = \mathrm{B}_\varphi(v rs) =  \mathrm{B}_\varphi(v s)\cup\mathrm{B}_\varphi(v r)$ and so the result follows. 
\end{proof}

The proof of Lemma \ref{Jdescription} shows that the $\{s,r\}$ braid relations of the Coxeter group either always preserve a tiling or always alter it, regardless of the choice of reduced words.
\begin{corollary}\label{Lemma Jphi is all or nothing}
Let $\varphi: W \hookrightarrow \mathrm{Sym}(n)$ be an  E-embedding. Let $r,s \in S$, then $\{r,s\} \notin J_{\varphi}$ if and only if for all reduced words containing a consecutive subword of $\underbrace{srs\ldots }_{m_{s,r}}$, we have
$$\mathrm{T}_\varphi(s_{i_1}\ldots s_{i_\ell }\underbrace{srs\ldots }_{m_{s,r}} s_{i_r}\ldots s_{i_j}) \ne \mathrm{T}_\varphi(s_{i_1}\ldots s_{i_\ell }\underbrace{rsr\ldots }_{m_{s,r}} s_{i_r}\ldots s_{i_j}).$$
\end{corollary}

It is now evident that the necessary bijection for Theorem~\ref{maintheorem} follows.

\end{document}